\documentclass[12pt,a4paper]{article}
\usepackage[a4paper, textheight=24cm,textwidth=16cm]{geometry}

\usepackage{amssymb}
\usepackage{amstext}
\usepackage{amsmath}
\usepackage{amsthm}
\usepackage{amscd}
\usepackage{tikz-cd}
\usepackage{color}
\usepackage{url}
\usepackage{mathrsfs}
\usepackage{hyperref}
\usepackage[OT2,T1]{fontenc}
\usepackage[french, english]{babel}
\DeclareMathOperator{\Spec}{Spec\,}
\DeclareMathOperator{\Res}{Res}
\DeclareMathOperator{\Resh}{Res^h}

\newcommand{\im}{\mathrm{Im}}
\newcommand{\rN}{\mathrm{N}}
\newcommand{\inj}{\hookrightarrow}
\newcommand{\flis}{\xrightarrow{\sim}}
\newcommand{\ffl}{\longrightarrow}
\newcommand{\CC}{{\mathbb C}}
\newcommand{\ZZ}{{\mathbb Z}}
\newcommand{\NN}{{\mathbb N}}
\newcommand{\AAA}{{\mathbb A}}
\def\Fp{\mathfrak{p}}
\newcommand{\ol}{\overline}
\newcommand{\OOO}{{\mathscr{O}}}
\newcounter{nc}
\renewcommand{\thenc}{{\rm(\roman{nc})}}
\newenvironment{romlist}%
{\begin{list}{\thenc}{
\usecounter{nc}
\parsep=0pt
\setlength  \labelwidth{\leftmargin}
\addtolength\labelwidth{-\labelsep}
}
}{\end{list}}
\newcommand{\pauseromlist}%
{\global\edef\savecount{\arabic{nc}}\end{romlist}}
\newcommand{\finpauseromlist}%
{\begin{romlist}\setcounter{nc}{\savecount}}
\newcounter{nnc}
\renewcommand{\thennc}{{\rm(\alph{nnc})}}
{\begin{list}{\thennc}{
\usecounter{nnc}
\parsep=0pt
\setlength  \labelwidth{\leftmargin}
\addtolength\labelwidth{-\labelsep}
}
}{\end{list}} 
{\begin{list}{$\bullet$}{\parsep=0pt}}{\end{list}} 
\newcounter{ctnum}
\renewcommand{\thectnum}{\textup{(\arabic{ctnum})}}
\newenvironment{numlist}%
{\begin{list}{\thectnum}{
\usecounter{ctnum} 
\parsep=0pt
\leftmargin=0pt%
\setlength{\itemindent}{\labelwidth}%
\addtolength{\itemindent}{\labelsep}%
}
}{\end{list}}

\newcounter{ssl}
\renewcommand{\thessl}{\thesubsection.\arabic{ssl}}
\newenvironment{subseclist}%
{\begin{list}{{\bfseries\thessl.}\hfill}{
\usecounter{ssl}
\leftmargin=0pt
\rightmargin=0pt
\labelwidth=2em 
\labelsep=.5em
\itemindent=2.5em
\itemsep=1ex
\parsep=0pt
\listparindent=2em%
}
}{\end{list}}

\newtheorem{enMBteo}[subsection]{Theorem}
\newtheorem{enMBprop}[subsection]{Proposition}

\newtheorem{enMBdefi}[subsection]{Definition}
\theoremstyle{definition}

\newtheorem{enMBsnot}[subsubsection]{Notation}

\newtheorem{enMBsubrems}[subsubsection]{Remarks}

\theoremstyle{plain}

\newtheorem{enMBsubprop}[subsubsection]{Proposition}
\newtheorem{enMBsublem}[subsubsection]{Lemma}

\newtheorem{enMBsubdefi}[subsubsection]{Definition}

\newcommand\rref[1]{{\rm\ref{#1}}}

\DeclareMathOperator{\rad}{rad}
\newcommand{\moins}{\smallsetminus}

\newcommand{\ul}[1]{\underline{#1}}
\newcommand{\acc}[1]{\{{#1}\}}

\newcommand{\hens}[1]{{#1}^{\mathrm{h}}}
\newcommand{\fini}[1]{{#1}^{\mathrm{f}}}

\newcommand{\algh}[1]{\hens{\mathrm{Alg}}_{#1}}
\newcommand{\algp}[1]{\mathrm{Alg}^{+}_{#1}}
\title{A henselian preparation theorem\\
}
\author{%
Laurent Moret-Bailly\thanks{Univ Rennes, CNRS, IRMAR - UMR 6625, F-35000 Rennes, France}
 \thanks{\tt laurent.moret-bailly@univ-rennes1.fr}
\thanks{The author was partially supported by the Centre Henri Lebesgue (program ANR-11-LABX-0020-0) and the  Geolie project (ANR-15-CE 40-0012).}}

\begin{document}
\selectlanguage{english}
\date{\today}
%
\maketitle
\begin{center}
\fbox{ 
\parbox{12cm}{\begin{center}Published in {\sl Israel Journal of Mathematics},\\ {\bf257} (2023), 519--531
(volume in honor of Moshe Jarden)\\
doi: 10.1007/s11856-023-2545-1
\end{center}}%
}\\
\vskip1.5cm
{\sl Respectfully dedicated to Moshe Jarden}\end{center}
\medskip

\begin{abstract}
We prove an analogue of the Weierstrass preparation theorem for henselian pairs, generalizing the local case recently proved by Bouthier and \v{C}esnavi\v{c}ius. As an application, we construct a henselian analogue  of the resultant of $p$-adic series defined by Berger.\end{abstract}

\section{Introduction}
Let $R$ be a ring (commutative, with unit). We denote by $R\acc{t}$ the henselization of the polynomial ring $R[t]$ with respect to the ideal $(t)$: this is a subring of the power series ring $R[[t]]$. (For a brief review of henselian pairs and henselization, see Section \ref{ssecHensPairs}).

The aim of this work is to prove the following result:

\begin{enMBteo}\label{ThIntro}
Let $R$ be a ring, $I$ an ideal of $R$. Assume that $(R,I)$ is a henselian pair. Let $d$ be a natural integer and let $f$ be an element of $R\acc{t}$ which in $R[[t]]$ has the form $f= \sum_{i\geq0}a_{i}t^{i}$, where $a_{d}\in R^\times$ and $a_{i}\in I$ for $i<d$. Then:
\begin{numlist}
\item\label{ThIntro1}  The images of $1,t,\dots,t^{d-1}$ form a basis of the  $R$-module $S=R\acc{t}/(f)$. 
\item\label{ThIntro2}  \textup{(Division theorem)} Every element of $R\acc{t}$ can be written uniquely in the form $Bf+C$  where $B\in R\acc{t}$ and where $C\in R[t]$ is a polynomial of degree $<d$. 
\item\label{ThIntro3} \textup{(Preparation theorem)} $f$ can be written uniquely as $f=(t^d+Q)\,v$ where $v\in R\acc{t}^\times$ and where $Q\in R[t]$ has degree $<d$ and coefficients in  $I$. 
\end{numlist}
\end{enMBteo}

\subsection{Related results}

The result (today) most commonly named ``Weierstrass preparation theorem'' is the analogous statement where $R\acc{t}$ is replaced by $R[[t]]$ where $R$ is a complete noetherian local ring with maximal ideal $I$: see for instance  \cite[VII, \S3, \no 8, prop. 5]{BourAC5-7}. 
This formal variant was generalized by  O'Malley \cite[2.10]{OMalley72} to the case where $R$ is $I$-adically complete and separated  (but is no longer assumed local or noetherian).

In the local case, there is a convergent variant, where $R=K\langle x_{1},\dots,x_{n}\rangle$ is the ring of germs of analytic functions in $n$ variables over some  field $K$ complete for an absolute value, and the role of $R\acc{t}$ is played by $K\langle x_{1},\dots,x_{n},t\rangle$. For $K=\CC$, this is in fact the original theorem of Weierstrass. It is generally proved by inspection of the above formal variant (where $R$ is $K[[ x_{1},\dots,x_{n}]]$), checking that the series constructed in the proof remain convergent; see for instance \cite[Theorem 45.3]{NagLocal}.

When $R$ is \emph{local henselian} with maximal ideal $I$, Theorem \ref{ThIntro} was proved by Bouthier and \v{C}esnavi\v{c}ius in \cite[3.1.2]{Bou-Ces20}, which inspired the present paper.  The proof we give here is somewhat different and more direct: we do not use reduction to the noetherian case or the classical preparation theorem, but we work directly from the construction of $R\acc{t}$ as a filtered colimit of \'etale $R[t]$-algebras.

Regrettably, there does not seem to be, at the moment, a general result covering all the above-mentioned variants, or at least a common strategy of proof.

\subsection{Outline of the paper}
In Section \ref{sec:Prelim}, we recall some basic facts about henselian pairs and henselization, some elementary results on henselian series rings (i.e. of the form $R\acc{t_{1},\dots,t_{n}}$), and a useful decomposition result for $R$-schemes, where $R$ is as in Theorem \ref{ThIntro}. 

Theorem  \ref{ThIntro} itself is proved in section \ref{sec:ThPrep}. The three statements are easily deduced from each other; here we derive  \ref{ThIntro2} and \ref{ThIntro3} from \ref{ThIntro1}. 

Finally, as an easy application, we define in Section \ref{sec_resultant} a notion of resultant in $R\acc{t}$, entirely similar to the resultant constructed by Berger \cite{Berger20} for $p$-adic formal power series.

\bigskip

\noindent {\bf Notation and conventions.} 
All rings are commutative with unit; ring homomorphisms respect unit elements. The unit group of a ring $R$ is denoted by $R^\times$,  its Jacobson radical by $\rad(R)$.

If $x$ is a point of a scheme, $\kappa(x)$ denotes its residue field.

Let $Y$ be a closed subscheme of a scheme $X$. We say $(X,Y)$ is a \emph{Zariski pair} if $X$ is the only open subscheme of $X$ containing $Y$; this condition only depends on the underlying spaces. If $X=\Spec(A)$ is affine and $I\subset A$ is the ideal of $Y$, we say $(A,I)$ is a Zariski pair if $(X,Y)$ is a Zariski pair or, equivalently, if $I\subset\rad(A)$. If $(X,Y)$ is Zariski and $X'\to X$ is a closed morphism, then $(X',Y\times_{X}X')$ is Zariski.
\bigskip

\noindent {\bf Acknowledgments.} The author is grateful to Henri Lombardi and Herwig Hauser for pointing out references, and to the referee for his/her remarks.

\section{Preliminary results}\label{sec:Prelim}
\subsection{Review of henselian pairs}\label{ssecHensPairs}
The notion of a henselian pair was defined by  Lafon \cite{L}, generalizing the local case introduced by Azumaya \cite{Az}. 
Let us first recall the definition:
\begin{enMBsubdefi}
Let $R$ be a ring and $I$ an ideal of $R$. We say that $(R,I)$ is a \emph{henselian pair} if for every \'etale $R$-algebra $R'$, every morphism $\ol{\rho}:R'\to R/I$ of $R$-algebras lifts to a morphism $\rho:R'\to R$.
\end{enMBsubdefi}
If $(R,I)$ is a henselian pair, we also say occasionally that $(\Spec(R), \Spec(R/I))$ is a henselian pair. (There is an obvious genaralization to general schemes, but we only need the affine case). 
A \emph{henselian local ring} is a local ring $R$, with maximal ideal $I$, such that $(R,I)$ is henselian.

A henselian pair is a Zariski pair: if $f\in 1+I$, apply the definition to $R'=R_{f}$. It follows that, given $\ol{\rho}$ as in the definition, $\rho$ is unique. Another immediate consequence of the henselian property is that the map $R\to R/I$ induces a bijection on idempotents: consider $R'=R[x]/(x(x-1))$.

There are many equivalent definitions of a henselian pair; for this and for more generalities, see for instance \cite[Tag 09XD]{StProj}. One important property that we shall use  is that if $(R,I)$ is a henselian pair, so is  $(R',IR')$ for every \emph{finite} (or just integral) $R$-algebra $R'$. In particular, idempotents of $R'/IR'$ lift uniquely to idempotents of $R'$.

\subsubsection{Henselization}
Let $R$ be a ring and $I\subset R$ an ideal. The category of henselian pairs $(S,J)$, where $S$ is an $R$-algebra and $J$ is an ideal containing $IS$, has an initial object $\hens{(R,I)}=(\hens{R},\hens{I})$%
\footnote{Of course, the notation $\hens{R}$ will be used only if there is no doubt about $I$.}
 called the \emph{henselization} of $(R,I)$ (or the henselization of $R$ at $I$). We have $\hens{I}=I\hens{R}$ and $R/I\flis\hens{R}/\hens{I}$. We can construct $\hens{R}$ as the filtered colimit of \'etale $R$-algebras $R'$ such that $R/I\flis R'/IR'$; in particular, $\hens{R}$ is flat over $R$, and faithfully flat if $(R,I)$ is a Zariski pair. If $R'$ is an integral $R$-algebra (for instance a quotient of $R$), then $\hens{(R',IR')}=\hens{(R,I)}\otimes_{R}R'$.

\subsection{Structure of henselian series rings}\label{sec:StructureSeries}
Let $R$ be a ring, $\ul{t}=(t_{1},\dots,t_{n})$ a finite sequence of indeterminates
\footnote{For the preparation theorem we only need the case  $n=1$. The case of an infinite set of indeterminates is left to the reader.} 
We denote by $R\acc{\ul{t}}$ the henselization of $R[\ul{t}]$ at the ideal $(t_{1},\dots,t_{n})$; it is an $R[\ul{t}]$-algebra with an isomorphism $\varepsilon:R\acc{\ul{t}}/(\ul{t})\flis R$, and there is a natural injection $R\acc{\ul{t}}\inj R[[\ul{t}]]$ making $R[[\ul{t}]]$ the  $(\ul{t})$-adic completion of $R\acc{\ul{t}}$; the image of $f\in R\acc{\ul{t}}$ in $R[[\ul{t}]]$ will be denoted by $f_{\mathrm{for}}$. 

As a functor of $R$, $R\acc{\ul{t}}$ is better behaved than $R[[\ul{t}]]$. In particular, it commutes with filtered colimits, and if $I$ is any ideal of $R$ we have $R\acc{\ul{t}}/IR\acc{\ul{t}}\cong(R/I)\acc{\ul{t}}$.

For $f\in R\acc{\ul{t}}$ we have the equivalences:  
$$f\in R\acc{\ul{t}}^\times\Leftrightarrow f_{\mathrm{for}}\in R[[\ul{t}]]^\times\Leftrightarrow \varepsilon(f)\in R^\times.$$
It follows that $\rad(R\acc{\ul{t}})$ is generated by  $\rad(R)$ and $(\ul{t})$. In particular, if  $(R,I)$ is a Zariski pair, so is $(R\acc{\ul{t}}, I R\acc{\ul{t}}+(\ul{t}))$. 

Similarly, if $(R,I)$ is a henselian pair, so is $(R\acc{\ul{t}}, I R\acc{\ul{t}}+(\ul{t}))$: to see this, view $R$ as the quotient $R\acc{\ul{t}}/(\ul{t})$ and apply the transitivity property \cite[0DYD]{StProj}.

Classically, $R\acc{\ul{t}}$ can be constructed as the colimit of a filtered family  $(A_{\lambda})_{\lambda\in L}$ of \'etale $R[\ul{t}]$-algebras, with compatible isomorphisms $\varepsilon_{\lambda}:A_{\lambda}/(\ul{t})A_{\lambda}\flis R$. In particular, for all $\lambda\in L$ and  $N\in\NN$, the natural morphism of  $R$-algebras
$ R[\ul{t}]/(\ul{t})^N \to A_{\lambda}/(\ul{t})^N A_{\lambda}$
is an isomorphism.

Each natural  morphism $\pi_{\lambda}:\Spec(A_{\lambda})\to\Spec(R)$ is smooth of relative dimension $n$, and has a section $s_{\lambda}$ deduced from $\varepsilon_{\lambda}$. 

We say that an $R$-algebra $A$ is \emph{geometrically irreducible} if the natural morphism  $\Spec(A)\to\Spec(R)$ has geometrically irreducible fibers.

\begin{enMBsublem}\label{lem:StrucSerHens} Let $R$ and  $\ul{t}=(t_{1},\dots,t_{n})$ be as above. Then one can choose the system $(A_{\lambda})_{\lambda\in L}$ such that each $A_{\lambda}$ is a geometrically irreducible $R$-algebra.
\end{enMBsublem}
\begin{proof} Starting with an arbitrary family $(A_{\lambda})_{\lambda\in L}$, we may assume, by enlarging $L$, that for all $\lambda\in L$ and $f\in A_{\lambda}$ such that $\varepsilon_{\lambda}(f)\in R^\times$, the localized algebra $A_{\lambda}[1/f]$ is still in the family. It suffices to show that, assuming this, the sub-system formed by the geometrically irreducible $A_{\lambda}$'s is cofinal. For each $\lambda$, let $U_{\lambda}\subset\Spec(A_{\lambda})$ be the union of the connected components of the fibers of  $\pi_{\lambda}$ meeting the section  $s_{\lambda}$. As $\pi_{\lambda}$ is smooth, $U_{\lambda}$ is open in  $\Spec(A_{\lambda})$ \cite[(15.6.7)]{EGA4_III}, and its fibers over $\Spec(R)$ are smooth and connected, with a rational point, hence geometrically irreducible. Since $U_{\lambda}$ is open, there is $f\in A_{\lambda}$ such that $\im(s_{\lambda})\subset \Spec(A_{\lambda}[1/f])\subset U_{\lambda}$ (in an affine scheme, every closed subset has a basis of principal open neighborhoods). The fibers of  $\Spec(A_{\lambda}[1/f])\to\Spec(R)$ are nonempty and open in those of  $U_{\lambda}\to\Spec(R)$ and therefore geometrically irreducible. This completes the proof.\end{proof}

\subsubsection{Evaluation}\label{ssec:eval}
This section will not be used until Section \ref{sec_resultant}. 

Let us keep the notation of  \ref{sec:StructureSeries} and consider the category  $\algh{R}$ of henselian pairs  $(A,J)$ where $A$ is an $R$-algebra. Then $(R\acc{\ul{t}},(\ul{t}))$ is an object of $\algh{R}$ corepresenting the set-valued functor $(A,J)\mapsto \prod_{i=1}^n J$. In particular, for an object $(A,J)$ of $\algh{R}$  and a sequence $\ul{\alpha}=(\alpha_{1},\dots,\alpha_{n})$ from $J$, we have a morphism ``evaluation at $\ul{\alpha}$'' from $R\acc{\ul{t}}$ to $A$ which we denote by $f\mapsto f(\ul{\alpha})$. One may construct it by noting that the morphism  $P\mapsto P(\ul{\alpha})$ from $R[\ul{t}]$ to $A$ maps the  $t_{i}$'s into $J$, hence factors through $R\acc{\ul{t}}$ because  $(A,J)$ is henselian.

For given $\ul{\alpha}$, the element $f(\ul{\alpha})$ is the sum in $A$, for the $J$-adic topology, of the series $f_{\mathrm{for}}(\ul{\alpha})$ obtained by substituting $\ul{\alpha}
$ for $\ul{t}$; this property characterizes  $f(\ul{\alpha})$ if $A$ is $J$-adically separated (but not in general).

The reader can check the following nice property, which will not be used here: if $I$ is an ideal of  $R$ generated by $n$ elements $a_{1},\dots,a_{n}$, the evaluation morphism $f\mapsto f(\ul{a})$ induces an isomorphism of  $R\acc{\ul{t}}/(t_{i}-a_{i})_{1\leq i\leq n}$ with the henselization $\hens{(R,I)}$.

\subsection{Schemes over henselian pairs: a decomposition result}\label{ssecQuasiFini}
\begin{enMBsnot}\label{not:qf1} Let $(R,I)$ be a henselian pair. Put $S=\Spec(R)$,  $\ol{R}=R/I$, and $\ol{S}=\Spec(\ol{R})$; more generally, for each $R$-algebra $A$, (resp.\ each $R$-scheme $X$) we shall put $\ol{A}=A/IA$ (resp.\ $\ol{X}=X\times_{S}\ol{S}$).
\end{enMBsnot}
The following proposition is a variant of \cite[XI, cor. 1 p. 119]{RayHens}:

\begin{enMBsubprop}\label{propQuasiFini} 
With notation as above, let $Z$ be a separated $R$-scheme of finite type. Assume that  $\ol{Z}$ is \emph{finite} over $\ol{R}$. 

Then there is a unique open and closed subscheme  $\fini{Z}$ of $Z$ which is finite over $R$ and satisfies $\ol{\fini{Z}}=\ol{Z}$. Moreover $\fini{Z}$ has the following properties:
\begin{numlist}
\item\label{propQuasiFini1} The pair $(\fini{Z},\ol{Z})$ is henselian.
\item\label{propQuasiFini2}  $\fini{Z}$ is the smallest open subscheme of  $Z$ containing $\ol{Z}$.
\item\label{propQuasiFini3} Let  $T$ be an $R$-scheme and  $u:T\to Z$ an $R$-morphism. Assume that  $(T,\ol{T})$ is a Zariski pair. Then $u$ factors through $\fini{Z}$.
\end{numlist}
\end{enMBsubprop}
\begin{proof}
Let us first assume the existence of $\fini{Z}$ and prove  \ref{propQuasiFini1},  \ref{propQuasiFini2} and  \ref{propQuasiFini3}. First, \ref{propQuasiFini1} is clear since  $(R,I)$ is henselian and  $\fini{Z}$ is finite over $R$. In particular, $(\fini{Z},\ol{Z})$ is a Zariski pair, and  \ref{propQuasiFini2} follows because  $\fini{Z}$ is open in  $Z$. Now take $u:T\to Z$  as in \ref{propQuasiFini3}: then $u^{-1}(\fini{Z})$ is a neighborhood of $\ol{T}$ in $T$, hence equal to $T$, which proves \ref{propQuasiFini3}.

Observe that \ref{propQuasiFini2}, for instance, implies the uniqueness of $\fini{Z}$. 
Now let us prove existence. First, consider the set $Z'$ of points $x\in Z$  isolated in their fiber above  $\Spec(R)$. Then $Z'$ is open in $Z$ \cite[(13.1.4)]{EGA4_III} and, viewed as an open subscheme, it is quasifinite over $\Spec(R)$; in addition, we have $\ol{Z'}=\ol{Z}$. So it is clear that if $\fini{Z'}$ exists it is  open in  $Z$, and closed since it is finite over $R$, so we can take $\fini{Z}=\fini{Z'}$. Replacing  $Z$ by $Z'$, we can therefore assume  $Z$ \emph{quasifinite} over $R$.

By Zariski's main theorem  \cite[(18.12.13)]{EGA4_IV}, there is an open  immersion  $Z\inj Z^{c}$, where $Z^{c}$ is a finite $R$-scheme. As $\ol{Z}$ is finite over $R$, the induced open immersion $\ol{Z}\inj\ol{Z^{c}}$ is closed, so we have $\ol{Z^{c}}=\ol{Z}\amalg Y$ for an open and closed subscheme $Y$ of $\ol{Z^{c}}$. Since $(R,I)$ is henselian and $Z^{c}$ is finite over $R$, this decomposition is induced (using the idempotent lifting property) by a decomposition $Z^{c}=\fini{Z}\amalg Z^{c}_{1}$ of $Z^{c}$, where $\fini{Z}$ and $Z^{c}_{1}$ are finite over $R$ and $\ol{\fini{Z}}=\ol{Z}$. In particular  $(\fini{Z},\ol{\fini{Z}})=(\fini{Z},\ol{Z})$ is a Zariski pair. Since $Z\cap \fini{Z}$ is open in  $\fini{Z}$ and contains $\ol{Z}$, it is therefore equal to $\fini{Z}$ which means that $\fini{Z}\subset Z$ and $Z=\fini{Z}\amalg Z'$ with $Z':=Z\cap Z^{c}_{1}$. Thus, the desired conditions for $\fini{Z}$ are satisfied.
\end{proof}

\begin{enMBsubrems}\label{RemFonctQuasiFini} \begin{numlist}
\item Assertions \ref{propQuasiFini2} and \ref{propQuasiFini3} of  \ref{propQuasiFini} only use the existence of  $\fini{Z}$ and the Zariski property for $(R,I)$.
\item We see in particular that $\fini{Z}$ is the largest closed subscheme of  $Z$ which is finite over  $S$. Moreover, $\fini{Z}$ is functorial in  $Z$: if $Y$ is a separated  $R$-scheme of finite type with  $\ol{Y}$ finite over $\ol{R}$, every $R$-morphism  $Z\to Y$ sends $\fini{Z}$  to $\fini{Y}$.
\item Using more sophisticated tools, one can generalize  \ref{propQuasiFini} by replacing ``finite'' by ``proper'' in the conditions for $\ol{Z}$ and $\fini{Z}$. For the proof, the first step (reduction to the quasifinite case) is of course ignored. One uses Nagata compactification to choose an open immersion $Z\inj Z^{c}$ into a proper $S$-scheme $p:Z^{c}\to S$. Then by the properness of $Z^c$ and the henselian property of $(R,I)$, we can apply \cite[Tag 0A0C]{StProj} to the sheaf $(\ZZ/2\ZZ)_{Z^c}$ to conclude that the idempotent defining $\ol{Z}$ in $\ol{Z^{c}}$ lifts to a unique idempotent on $Z^c$, which we take to define $\fini{Z}$.
\item Assume that $R$ is local henselian and $I$ is its maximal ideal, and let $Y$ be a separated $R$-scheme of finite type. Let $y$ be an \emph{isolated} point of $\ol{Y}$. Then $C:=\ol{Y}\moins\{y\}$ is closed in $Y$, so we can apply \ref{propQuasiFini} to $Z:=Y\moins C$ since $\ol{Z}=\{y\}$ set-theoretically. It is then easy to see that $\fini{Z}=\Spec(\OOO_{Y,y})$. In particular, $\OOO_{Y,y}$ is a finite $R$-module: this is the Mather division theorem as stated in \cite[Theorem 1]{AlLo2010}. The approach in \cite{AlLo2010} (and the related paper \cite{AlLo2017}) is algorithmic, while here we use Zariski's main theorem as a magic wand.
\end{numlist}\end{enMBsubrems}

\section{The preparation theorem}\label{sec:ThPrep}
\subsection{Notation and assumptions}\label{ssec:ThPrepNot}
We fix a ring $R$ and an indeterminate $t$. We denote by $\algp{R[t]}$ the category of pairs $(A,x)$ where $A$ is an $R[t]$-algebra and  $x$ is an element of  $A$.

We also fix an element $f$ of $R\acc{t}$, and we write $$f_{\mathrm{for}}=\sum_{i\geq0}a_{i}t^i \in R[[t]]\quad(a_{i}\in R).$$
We assume that \emph{the ideal generated by the  $a_{i}$'s ($i>0$) is equal to $R$}. Equivalently, for all $\Fp\in\Spec(R)$, the image of $f$ in $\kappa(\Fp)[[t]]$, or in $\kappa(\Fp)\acc{t}$, is not a constant. 

Finally we denote by $S$ the $R[t]$-algebra $R\acc{t}/(f)$.

\begin{enMBprop}\label{prop:platitude} With the assumptions of  \rref{ssec:ThPrepNot}, we also fix an indeterminate $u$.
\begin{numlist}
\item\label{prop:platitude1} The object $(R\acc{t},f)$ of $\algp{R[t]}$ is the filtered colimit of a system $(A_{\lambda},f_{\lambda})_{\lambda\in L}$ with, for each  $\lambda\in L$,  the following properties:
\begin{romlist}
\item\label{prop:platitude11} The $R[t]$-algebra $A_{\lambda}$ is \'etale and, for all $n\in\NN$, the canonical morphism  $R[t]/(t^{n})\to A_{\lambda}/t^{n}A_{\lambda}$ is an isomorphism.
\item\label{prop:platitude13} The canonical  $R$-morphism $R[u]\to A_{\lambda}$ mapping $u$ to $f_{\lambda}$ is flat and quasifinite.
\end{romlist}
In particular, the canonical $R$-morphism $R[u]\to R\acc{t}$ mapping $u$ to $f$ is flat, and $f$ is a nonzerodivisor in $R\acc{t}$.
\item\label{prop:platitude2} The $R[t]$-algebra $S$ is the filtered colimit of a system $(S_{\lambda})_{\lambda\in L}$ with the following properties:
\begin{romlist}
\item\label{prop:platitude21} Each $R$-algebra $S_{\lambda}$ is flat, of finite presentation and quasifinite, and the transition maps $S_{\lambda}\to S_{\mu}$ ($\lambda\leq\mu)$ are \'etale. (In particular, $S$ is flat over $R$.)
\item\label{prop:platitude22} For all $n\in\NN$ and $\lambda\in L$, the canonical morphism  $R[t]/(t^{n})\to S_{\lambda}/t^{n}S_{\lambda}$ is surjective.
\end{romlist}
\end{numlist}
\end{enMBprop}
\begin{proof}
Part  \ref{prop:platitude1} immediately implies part \ref{prop:platitude2}, with $S_{\lambda}=A_{\lambda}/(f_{\lambda})$ (the transition maps are \'etale due to the same property for the $A_{\lambda}$'s, which are \'etale over $R[t]$).

To prove \ref{prop:platitude1}, write $R\acc{t}=\varinjlim_{\lambda\in L}A_{\lambda}$ as in Lemma \ref{lem:StrucSerHens}, and call $t_{\lambda}\in A_{\lambda}$ the canonical image of  $t$. There exists $\lambda_{0}\in L$ and  $f_{\lambda_{0}}\in A_{\lambda_{0}}$ mapping to $f$; we can restrict  $L$ to the indices  $\lambda\geq\lambda_{0}$ and, for each  $\lambda$, denote by $f_{\lambda}\in A_{\lambda}$ the image of $f_{\lambda_{0}}$. Clearly,  we have $(R\acc{t},f)=\varinjlim_{\lambda\in L}(A_{\lambda},f_{\lambda})$. Part  \ref{prop:platitude1}\,\ref{prop:platitude11} is obvious from the choice of $(A_{\lambda})_{\lambda\in L}$.

Let us prove \ref{prop:platitude1}\,\ref{prop:platitude13}. For fixed $\lambda$, we can view $f_{\lambda}$ as a  morphism $g_{\lambda}:X_{\lambda}:=\Spec(A_{\lambda})\to\AAA^1_{R}=\Spec(R[u])$ of $R$-schemes. For  $s\in\Spec(R)$, the $\kappa(s)$-morphism  $g_{\lambda,s}:X_{\lambda,s}\to\AAA^1_{\kappa(s)}$ induced on the fibers is deduced from $1\otimes f_{\lambda}\in\kappa(s)\otimes_{R}A_{\lambda}$, whose image in  $\kappa(s)\otimes_{R}R\acc{t}$ is assumed nonconstant. So $g_{\lambda,s}$ is not constant  on  $X_{\lambda,s}$, which is a smooth geometrically irreducible curve over $\kappa(s)$. It follows that  $g_{\lambda,s}$ is flat and quasifinite. Since $X_{\lambda}$ and $\AAA^1_{R}$ are smooth over  $\Spec(R)$, the ``fiberwise flatness'' criterion \cite[(11.3.10)]{EGA4_III} shows that  $g_{\lambda}$ is flat. It is also quasifinite since it is affine of finite presentation with finite fibers. This completes the proof.\end{proof}

\begin{enMBdefi}\label{def_weier} 
Let $R$ be a ring, $I$ an ideal of  $R$, $t$ an indeterminate.

We say that a formal power series $f=\sum_{i\geq0}a_{i}t^i\in R[[t]]$ is  \emph{$I$-normal} if there is   $d\in\NN$ such that $a_{d}\in R^\times$ and $a_{i}\in I$ for $i<d$. The integer $d$ (unique if $I\neq R$) is called the order of $f$. 
 
 We say that $f$ is \emph{$I$-monic} of order $d$ if it is $I$-normal of order $d$ and $a_{d}=1$. 

An element $f$ of $R\acc{t}$ is $I$-normal ($I$-monic) of order $d$ if $f_{\mathrm{for}}\in R[[t]]$ is.
\end{enMBdefi}
\subsection{Proof of Theorem \ref{ThIntro}}\label{ssecDemThIntro}
As in \ref{ThIntro}, let $(R,I)$ be a henselian pair and let $f\in R\acc{t}$ be $I$-normal of order $d$, with $f_{\mathrm{for}}=\sum_{i\geq0}a_{i}t^i \in R[[t]]\quad(a_{i}\in R)$.  If $d=0$, everything is trivial, so we assume in addition that $d>0$; thus, the assumption of \ref{ssec:ThPrepNot} is satisfied and, in particular, Proposition \ref{prop:platitude} applies to $f$.

Assume assertion \ref{ThIntro}\,\ref{ThIntro1} is proved, i.e.\ $S=R\acc{t}/(f)$ is a free $R$-module with the images of $1,t,\dots,t^{d-1}$ as a basis. This immediately implies the division theorem \ref{ThIntro}\,\ref{ThIntro2}, with uniqueness coming from the fact that $f$ is a nonzerodivisor  (\ref{prop:platitude}\,\ref{prop:platitude1}). 
 
In turn, the division theorem implies the preparation theorem \ref{ThIntro}\,\ref{ThIntro3}. Indeed, the relation in \ref{ThIntro3} can be rewritten as $t^d=v^{-1}\,f-Q$, so that uniqueness follows from the uniqueness part of \ref{ThIntro2}; next, applying \ref{ThIntro2} to $t^d$, we find that $t^d=Bf-Q$ where $Q$ is a polynomial of degree $<d$. Reducing modulo $I$ and comparing coefficients, we see that $Q$ has coefficients in $I$ and the  constant term of $B$ is in $a_{d}+I$, which gives  \ref{ThIntro3} with $v=B^{-1}$.%

It remains to prove \ref{ThIntro}\,\ref{ThIntro1}. As in \ref{ssecQuasiFini}, we put $\ol{A}=A/IA$ for every $R$-algebra $A$. 

First we observe that the image $\ol{f}$ of $f$ in $\ol{R\acc{t}}\cong\ol{R}\acc{t}$ is the product of $t^d$ by a unit, so that $\ol{S}\cong\ol{R}\acc{t}/(t^d)\cong\ol{R}[t]/(t^d)$ which is $\ol{R}$-free with basis  $(1,t,\dots,t^{d-1})$. 

Let us write  $S$ as the colimit of a filtered system $(S_{\lambda})_{\lambda\in L}$ of $R[t]$-algebras with the properties of \ref{prop:platitude}\,\ref{prop:platitude2}. We have just seen that  $t^d$ vanishes in $\ol{S}$, so by changing the index set $L$ we may assume that  $t^d$  vanishes in $\ol{S_{\lambda}}$ for all $\lambda$: thus,  $\ol{S_{\lambda}}=\ol{S_{\lambda}/t^d S_{\lambda}}$ hence, by  \ref{prop:platitude}\,\ref{prop:platitude2}\,\ref{prop:platitude22}, it is a  quotient of $\ol{R}[t]/(t^d)$. So we have  morphisms of $\ol{R}[t]$-algebras $\ol{R}[t]/(t^d)\to\ol{S_{\lambda}}\to\ol{S}$ where the first map is surjective and the composition is an  isomorphism. We conclude that  $\ol{R}[t]/(t^d)\flis\ol{S_{\lambda}}$ for all $\lambda$. In particular, $\ol{S_{\lambda}}$ is finite over $\ol{R}$. As $(R,I)$ is henselian, we may apply Proposition \ref{propQuasiFini} and write $S_{\lambda}=\fini{S}_{\lambda}\times T_{\lambda}$, where $\fini{S}_{\lambda}$ is \emph{finite} over $R$ and $\ol{\fini{S}_{\lambda}}=\ol{S_{\lambda}}$. By functoriality (Remark \ref{RemFonctQuasiFini}), the quotients $\fini{S}_{\lambda}$ of the  $S_{\lambda}$'s form an inductive system. 

Since $S$ is a quotient of  $R\acc{t}$ and $(R\acc{t},IR\acc{t})$ is a Zariski pair, so is $(S,IS)$. Hence, for all $\lambda$, the canonical morphism  $S_{\lambda}\to S$ factors through $\fini{S}_{\lambda}$ by \ref{propQuasiFini}\,\ref{propQuasiFini3}, and finally $S=\varinjlim_{\lambda\in L}\fini{S}_{\lambda}$. 

Since, for each $\lambda$,  $S_{\lambda}$ is a flat  $R$-algebra of finite presentation, so is $\fini{S}_{\lambda}$, which is in addition a finite $R$-module, hence locally free. As $(1,t_{\lambda},\dots,t_{\lambda}^{d-1})$ induces an $\ol{R}$-basis of $\ol{\fini{S}_{\lambda}}$, and $I\subset\rad{(R)}$, it follows easily that $(1,t_{\lambda},\dots,t_{\lambda}^{d-1})$ is an  $R$-basis of $\fini{S}_{\lambda}$ for all $\lambda$, and part \ref{ThIntro1} follows.
\section{Application: a henselian resultant}\label{sec_resultant}

If  $R$ is a ring, $S$ a finite locally free $R$-algebra and $x$ an element of $S$, we denote by $\rN_{S/R}(x)\in R$ the norm of $x$ in $R$, i.e.\ the determinant of multiplication by $x$ in the  $R$-module $S$.

\begin{enMBdefi}\label{def_resultant} Let $(R,I)$ be a henselian pair. Let $f\in R\acc{t}$ be \emph{$I$-monic }of order $d$. Denote by $S$ the $R$-algebra $R\acc{t}/(f)$ \emph{(which is a free  $R$-module of rank $d$, by \rref{ThIntro}\,\rref{ThIntro1})}. 

For $g\in R\acc{t}$, the  \emph{(henselian) resultant} of $f$ and $g$, denoted by $\Resh(f,g)$, is the element of $R$ defined by
$$\Resh(f,g):=\rN_{S/R}(g).$$
\end{enMBdefi}

\subsection{Properties of the resultant}\label{ssec_PropRes}
We keep the notation and assumptions of  \ref{def_resultant}, and we denote by  $P=t^d+Q$ the polynomial associated to  $f$ by  \ref{ThIntro}\,\ref{ThIntro3}. The proofs of the following properties are easy and left to the reader.
 \begin{subseclist}
\item \emph{Functoriality:} Let $\varphi:(R,I)\to(R',I')$ be a  morphism of henselian pairs, $f'$ et $g'$ the images of $f$ and $g$ in $R'\acc{t}$. Then $\Resh(f',g')=\varphi(\Resh(f,g))$.
\item By construction, $\Resh(f,g)$ only depends on $f$  via the $R$-algebra $R\acc{t}/(f)$. In  particular, $\Resh(f,g)=\Resh(P,g)$.
\item $\Resh(f,g)$ only depends on  $g$ via its class modulo $f$; in other words, we have $\Resh(f,g+hf)=\Resh(f,g)$ for all $h\in R\acc{t}$. Moreover, $\Resh(f,g)\in R^\times$ if and only if the ideal $(f,g)\subset R\acc{t}$ equals  $R\acc{t}$. (More generally, see \ref{elim} below.)
\item\label{ssec_PropRes4}\emph{Special cases:}  If $\alpha\in R$, we have  $\Resh(f,\alpha)=\alpha^d$ and $\Resh(f,\alpha-t)=P(\alpha)$.

If $\alpha\in I$, then $\Resh(\alpha-t, g)=g(\alpha)$, and $\Resh(f,\alpha-t)=(1+\varepsilon)\,f(\alpha)$ for some $\varepsilon\in I$ by the second formula above (recall that $f$ is $I$-monic).
\item\label{ssec_PropRes5} \emph{Multiplicativities:} If $h\in R\acc{t}$, we have $\Resh(f,gh)=\Resh(f,g)\Resh(f,h)$; if in addition $h$ is $I$-monic of order $m$, then  $\Resh(fh,g)=\Resh(f,g)\Resh(h,g)$. For the second equality, one may use the exact sequence
$$0\ffl R\acc{t}/(h) \xrightarrow{\,\times f\,}R\acc{t}/(fh) \ffl R\acc{t}/(f) \ffl 0.$$
\item\label{ssec_PropResPol} \emph{Polynomials:} If $f$ and $g$ are in $R[t]$, with $f$ \emph{monic of degree $d$} (in the sense of polynomials), then $\Resh(f,g)$ is the usual resultant. The condition on $f$ is essential: for instance,  $\Resh(1+\alpha t,g)=1$ for all $\alpha\in R$ and $g\in R\acc{t}$. (In fact, for two possibly non-monic polynomials of respective degrees  $\leq d$ and $\leq m$, the definition of the classical resultant depends on the choice of $d$ and $m$.)
\item \emph{Weak symmetry:} Assuming that $g$ is $I$-monic of order $m$,  then $\Resh(g,f)= (-1)^{md}\,(1+\varepsilon)\,\Resh(f,g)$ for some $\varepsilon\in I$. To see this, reduce to the case of polynomials and apply \ref{ssec_PropResPol}.
\item\label{elim} \emph{Elimination:} Let $J\subset R\acc{t}$ be the ideal generated by $f$ and $g$. Then $\Resh(f,g)\in J$ (thus it belongs to $J\cap R$): indeed, in the free $R$-module  $S=R\acc{t}/(f)$, the image of multiplication by $g$ contains $\Resh(f,g)\,S$. 

Conversely, every $\alpha\in  J\cap R$ is a multiple of the class of $g$ in $S$ so, taking norms, $\alpha^d$ is a multiple of $\Resh(f,g)$ in $R$. In  particular, we have in  $R$ the  inclusions
$(\Resh(f,g))\subset J\cap R\subset\sqrt{(\Resh(f,g))}$. Geometrically, the closed subset $V(\Resh(f,g))\subset\Spec(R)$ is the projection of $V(f,g)\subset\Spec(R\acc{t})$.
\item \emph{Roots:} Let $\varphi:R\to R'$ be a ring homomorphism, and let $\alpha\in R'$ be a zero of $P$ in $R'$. First, I claim that $g(\alpha)$ makes sense in $R'$ and is  an element of $R[\alpha]\subset R'$. Indeed, 
the relation $P(\alpha)=0$ shows that (due to the form of $P$) $\alpha^d\in IR[\alpha]$, whence $\alpha\in\sqrt{IR[\alpha]}$. Since  $R[\alpha]$ is a finite $R$-module, the pair  $(R[\alpha],\sqrt{IR[\alpha]})$ is henselian, hence the claim.

Now assume that the image of   $P$ in $R'[t]$ factors as $\prod_{i=1}^d (t-\alpha_{i})$, where the $\alpha_{i}$'s are elements of  $R'$. Then we have  in $R'$ the equality
$$\varphi(\Resh(f,g))=\prod_{i=1}^d g(\alpha_{i})$$
as follows from the above remark and properties \ref{ssec_PropRes4} and \ref{ssec_PropRes5} (applied in the ring $R[\alpha_{1},\dots,\alpha_{d}]\subset R'$).

Note that if we assume for simplicity that $R=R'$ is a domain, then the $\alpha_{i}$'s are the zeros of $f$ in $\sqrt{I}$.
\item \emph{Power series:} Assume   $R$ is $I$-adically complete and separated. Then  $\Resh(f,g)=\Res(f_{\mathrm{for}},g_{\mathrm{for}})$ where $\Res$ denotes the resultant defined in \cite{Berger20}.
\end{subseclist}

 \end{document}